\documentclass[reqno,10pt]{amsart}

\usepackage{enumerate}
\usepackage{amsmath}
\usepackage{mathabx}
\usepackage[colorlinks=true]{hyperref}
\hypersetup{urlcolor=blue, linkcolor=blue, citecolor=red}

\newtheorem{thm}{Theorem}[section]
\newtheorem{lem}[thm]{Lemma}
\newtheorem{prop}[thm]{Proposition}

\newtheorem{exemp}[thm]{Example}
\theoremstyle{definition}
\newtheorem{defi}[thm]{Definition}
\newtheorem{rem}[thm]{Remark}

\newcommand\R{{\mathbb R}}
\newcommand\C{{\mathbb C}}

\newcommand\MScN[1]{\href{http://www.ams.org/mathscinet-getitem?mr=#1}{\nolinkurl{(#1)}}}
\newcommand\DOI[1]{\href{http://dx.doi.org/#1}{(doi: \nolinkurl{#1})}}
\newcommand\LINK[1]{\href{#1}{(link: \nolinkurl{#1})}}

\title[A generalized complex Ginzburg-Landau equation]{A generalized Complex Ginzburg-Landau Equation: global existence
and stability results}
\author{Sim\~ao Correia and M\'ario Figueira}

\begin{document}

\maketitle

\begin{abstract}
	We consider the complex Ginzburg-Landau equation with two pure-power nonlinearities and a damping term.
	After proving a general global existence result, we focus on the existence and stability of several periodic orbits, namely the trivial equilibrium, bound-states and solutions independent of the spatial variable. In particular, we construct bound-states either explicitly in the real line or through a bifurcation argument for a double eigenvalue of the Dirichlet-Laplace operator on bounded domains.
\end{abstract}



\section{Introduction and main results}\label{Intro}

The complex Ginzburg-Landau equation models various physical pheno-\break
mena especially in theory of superconductivity and fluid dynamics. A particular Ginzburg-Landau equation can be written:
\begin{equation}\label{eq:CGL}
\partial_t A = (1 + i\alpha) \Delta A + (1 + i \beta) |A|^2 A + k A 
\end{equation}
which admit the development of singularities for certain values of parameters (see e.g. 
 \cite{Caz-Weiss1, Masmou, Popp}. However,
the introduction of a high-order term with a negative sign, like $ - (1 + i c) |A|^4A$, allows 
to saturate the explosive instabilities. We refer e.g. \cite{Aronson} and \cite{Deissler} for a more complete physical background.

We are concerned with the study of a generalized Ginzburg-Landau equation
\begin{equation}
\tag{gCGL}
\label{gCGL}
\left\{\begin{array}{lll}
\partial_t u = (a + i\alpha) \Delta u + (b + i \beta) |u|^{\sigma_1} u - (c+i\gamma) |u|^{\sigma_2}u +k u, \\
B u(t, x) = 0,\,\,\, x\in \partial \Omega,\,\, t\geq 0 \\
u(0, x) = u_0(x) &
\end{array}\right.
\end{equation}
where $B$ is the identity operator (Dirichlet condition) or $B = \frac{\partial} {\partial n}$ (Neumann condition). We assume
 $ a >0, \alpha, b, \beta, k \in \R,  \sigma_1, \sigma_2 > 0$ and $\Omega$ as a domain in $\R^N$
 of class $C^2$ with $\partial \Omega$
bounded. If $c = \gamma = 0$, \eqref{gCGL} is reduced to the complex Ginzburg-Landau equation (CGL), 
equation widely studied under several assumptions
on the parameters since the seminal paper \cite{Lever}; see also \cite{GV1, GV2, Okaza} and 
the references therein. 

In this paper, we extend some results of global existence of solutions and their stability and also the existence of standing 
wave solutions in one dimension, previously exposed for the complex Ginzburg-Landau equation in \cite{Cor-fig}, where only one nonlinear term was present. Moreover, we prove the existence of standing waves in bounded domains through a bifurcation argument applied to double eigenvalues of the Dirichlet-Laplace operator, which is new even in the context of (CGL). As mentioned, the main interest in adding a higher-order term is the need for more precise physical descriptions.

Define the linear operators $-A_D = (a + i \alpha) \Delta,\, a > 0$, with domain
$ D(A_D) = H^2(\Omega)\cap H^1_0 (\Omega)$ (Dirichlet condition) and  $-A_N = (a + i \alpha) \Delta$, with domain
\begin{eqnarray*}
 \lefteqn{D(A_N) =  \{ u\in H^1(\Omega) :  \Delta u \in L^2(\Omega)  \,\, {\rm and} } \\
 & &  - \int_\Omega \Delta u \bar v = -\int_\Omega \nabla u \cdot \nabla \bar v\, dx , \,\, \forall v\in H^1(\Omega)  \} 
 \end{eqnarray*}
 (Neumann condition). It is well known that these operators generate an analytic semi-group (see \cite{Frid}).
Denoting by $A$ any of these two operators $A_D$ or $A_N$,
let us introduce the following definition:

\begin{defi}
A function $u(\cdot) \in C( [0, T); L^2(\Omega)), \,\, T>0$, is called a strong solution of \eqref{gCGL} if
$u(t)\in D(A)$, ${{du}\over{dt}} (t)$ exists for $t\in(0, T),\,\, u(0) = u_0$ and the differential
equation in \eqref{gCGL} is satisfied in $L^2(\Omega)$ for all $t\in (0, T)$.
\end{defi}

Since $f(u) = (b + i\beta)|u|^{\sigma_1} u - (c + i \gamma) |u|^{\sigma_2} u
 + k u$ is locally Lipschitz in $H^1(\Omega)$ with
values in $L^2(\Omega)$, for $\sigma_1, \sigma_2 \leq 2/(N-2)^+$ (with the convention $2/(N-2)^+=+\infty$ if $N=1,2$), then there exists $T=T(u_0) > 0$ such that the problem \eqref{gCGL}
has a unique solution on $[0, T_0)$, and this solution depends continuously of the
initial data (see \cite{Henry}, pag. 54 and 62). We begin with a global existence result:

\begin{thm}\label{existence}
Let $\Omega$ be a domain in $\R^N$ of class $C^2$ with $\partial \Omega$ bounded. Assume 
$0 < \sigma_j \leq 2/(N-2)^+$. 
Then, for any $u_0 \in H_0^1 (\Omega )$ (resp. $u\in H^1(\Omega)$), 
there exists $T = T(u_0) > 0$ such that \eqref{gCGL}
with $A = A_D$ (resp. $A = A_N$) has a unique strong solution on $[0, T)$ and this
solution depends continuously of the initial data. Moreover, if $ 0 < \sigma_1 < \sigma_2$,
 $c > 0,\ \alpha \neq 0 $ and $\gamma / \alpha \geq 0$,
 the solution is global.
\end{thm}

As expected, the lower order nonlinear term does not influence the global existence result. This proves in particular that the addition of a higher-order term with a specific sign prevents any possible blow-up mechanisms.

The existence of standing waves for the complex Ginzbourg-Landau equation remains a largely open problem. Before we proceed, we rewrite the generalized complex Ginzburg-Landau equation in its trigonometric form, following the notations of \cite{Caz-Weiss2} and \cite{Cor-fig}:
 
  \begin{equation}
  \tag{gCGL*}
  \label{gGLcomplex}
 u_t = e^{i\theta} \Delta u + e^{i \gamma_1} | u |^{\sigma_1}u + \chi  e^{i \gamma_2} | u |^{\sigma_2}u + ku
 \end{equation}
 
 \noindent where $-\pi/2 < \theta < \pi/2,\, -\pi<\gamma_1, \gamma_2\le \pi$, $ \,k\in \R,\, \chi = \pm 1.$
  \noindent Then one may look for solutions of \eqref{gGLcomplex} in the form $ u = e^{i\omega t} \phi (x) $, where
 $\phi \in H^1(\Omega)$ is a solution of the elliptic equation
 
 \begin{equation}\label{gGLelip1}
 \tag{B-S}
 i\omega \phi = e^{i \theta} \Delta \phi + e^{i \gamma_1} | \phi |^{\sigma_1} \phi + \chi  e^{i \gamma_2} | \phi |^{\sigma_2} \phi
 +  k \phi
 \end{equation}

For $k=0$,
the existence of standing wave solutions is already known in some particular cases : 
$\theta = \pm \gamma = \pm \pi/2$, which corresponds to the nonlinear Schr\"odinger equation
or $\omega = 0$ (stationary solutions). 
Outside of these cases, we refer \cite{Caz-Weiss2, Cor-fig, Puel}, where the
implicit function theorem is used to obtain the existence of standing waves of \eqref{gGLelip1} for $\chi=0$ and several constraints on the remaining
parameters. In \cite{Caz-Weiss2} it is proven that, for $\Omega$ bounded, the equation
\eqref{gGLelip1} (with $k=0$) has a solution $(\omega, u)\in \R\times H^1_0 (\Omega)$ bifurcating from $u= 0$ if $\sigma$ is sufficiently
small and $\cos \theta \cos \gamma > 0$. A similar result is obtained in \cite{Cor-fig}
where the aim was to trade the freedom in $k$ for the freedom in $\sigma$. The reference \cite{Puel} focuses on a bifurcation argument starting from the ground-state solution of the nonlinear Schr\"odinger equation, for both $\Omega$ bounded and the whole space (under some radial assumptions).

Our first result concerns an explicit bound-state in the real line.

 \begin{thm}\label{thmbs}
    Suppose $\Omega=\mathbb{R}$. Fix $-\pi/2 <\theta<\pi/2$ and $\omega, k\in \mathbb{R}$ such that $\omega\cos\theta + k\sin\theta\neq 0$. Define
	  \begin{equation}\label{eq:valor-d}
	d = \frac{ k \cos \theta - \omega\sin\theta  + \sqrt{\omega^2 + k^2}}
	{\omega\cos\theta + k \sin\theta}
	\end{equation}
	and let $\gamma_j\in (-\pi,\pi], \,(j=1, 2)$ be the unique solutions of
		\begin{equation}\label{eq:restricaogamma}
	\tan(\gamma_j-\theta)=\frac{d(\sigma_j+4)}{\sigma_j+2-2d^2},\quad d\sin(\gamma_j-\theta)+\cos(\gamma_j-\theta)>0.
	\end{equation}
	Then
	
	\noindent $a)$ If $\chi = 1$ the generalized complex Ginzburg-Landau equation 
	admits a bound-state of the form
	\begin{equation}\label{bsgGL}
	\phi=\psi\exp\left(id\ln\psi\right),
	\end{equation}
	where $\psi$ is the bound-state for the nonlinear Schr\"{o}dinger equation:
	\begin{equation}\label{est-schr1}
	\psi'' = \epsilon \psi - \eta_1 \psi^{\sigma_1+1} - \chi \eta_2 \psi^{\sigma_2+1}
	\end{equation}
	with
	$$ \epsilon = \frac{\sqrt{\omega^2+k^2}}{1+d^2},\quad
	\eta_j = \frac{d \sin (\gamma_j-\theta) + \cos (\gamma_j - \theta)}{1+d^2} $$
	
	\noindent $b)$ If $\chi = -1$, there exists a small enough $\delta > 0$ such that for
	$ 0 < \sigma_2 < \delta$ the generalized Ginzburg-Landau equation admits a bound-state of the form
	\eqref{bsgGL} with $\psi$  the bound-state for the nonlinear Schr\"{o}dinger equation satisfying
	\eqref{est-schr1}.
	
\end{thm}

\begin{rem}
	We observe that the conditions $\omega^2+k^2\neq 0$ and $\arg(k-i\omega)\neq\theta$ imply 
	$\omega\cos\theta + k\sin\theta\neq 0 $. 
\end{rem}

\begin{rem}
	In \cite{Cor-fig}, the uniqueness and stability of bound-states defined on $\mathbb{R}$ was studied. The same arguments may be applied in our framework without any extra difficulty.
\end{rem}

For $\Omega$ bounded, following the spirit of \cite{Caz-Weiss2, Cor-fig} for (CGL), we wish to construct solutions of \eqref{gGLelip1} through a bifurcation argument applied to the trivial solution $u\equiv0$. Therein, a bifurcation from simple eigenvalues of the Laplacian is built directly as an application of the Implicit Function Theorem. In the context of the \eqref{gGLelip1}, a similar procedure should be applicable. Instead, we turn our focus to the bifurcation problem for eigenvalues of multiplicity two, inspired in the methodology presented in \cite{Berger}. We remark that, even in the special case $\chi=0$, this is an open problem. A classic example where one has double eigenvalues is the case of the square $\Omega=(-1,1)^2$: we refer to \cite{DelPinoGarciaMusso} for a bifurcation result in this specific case. Our main result is the following:

\begin{thm}\label{teo:bifurca}
Given $\Omega\subset \R^N$ bounded and $2\le\sigma_1+1\le\sigma_2<4/(N-2)^+$, suppose that $\lambda_0$ is a double eigenvalue of the Dirichlet-Laplace operator with $L^2$-orthogonal real-valued eigenfunctions $u_1$, $u_2$. Suppose that the equation
$$
P(\alpha)=\int_\Omega |u_1+\alpha u_2|^{\sigma_1}(u_1+\alpha u_2)(\alpha u_1-u_2)=0
$$
has a solution $\alpha_0\in \C$ satisfying $P'(\alpha_0)\neq 0$. Then there exist $\delta>0$ and a Lipschitz mapping
$$
\epsilon\in [0, \delta) \to (y, \lambda, \alpha) \in H^1_0(\Omega)\times \C \times \C,
$$
with $(y(0), \lambda(0), \alpha(0))= (0,e^{i\theta}\lambda_0, \alpha_0)$ and $(y,u_1)_{L^2}=(y,u_2)_{L^2}=0$, such that
$$
u=y(\epsilon) + \epsilon u_1 + \epsilon \alpha(\epsilon) u_2
$$
is a solution to \eqref{gGLelip1} for $k-i\omega=\lambda(\epsilon)$. Moreover,
\begin{equation}
\lambda(\epsilon)=e^{i\theta}\lambda_0-e^{i\gamma_1}\epsilon^{\sigma_1} \int_\Omega |u_1+\alpha_0u_2|^{\sigma_1}(u_1+\alpha_0 u_2)u_1dx + o(\epsilon^{\sigma_1}).
\end{equation}
\end{thm}

\begin{rem}
	As a consequence of the above result, if $P$ does not have multiple roots,  the number of branches bifurcating at $e^{i\theta}\lambda_0$ is equal to the number of simple roots of $P$ (counting permutations of $u_1$ and $u_2$, see Example \ref{exemplo}).
\end{rem}


We now focus on the stability of the equilibrium solution $u \equiv 0$, the asymptotic
decay of the global solutions of \eqref{gCGL} depending on the parameters and the stability
of some particular time periodic solutions. To be more precise, we give the following definition:

\begin{defi}
	We say that the equilibrium point $u\equiv 0$ is $E$-stable if for any $\delta > 0$ there exists $\varepsilon > 0$ such that
	$$ u_0\in E , \,\, \| u_0 \|_{E} < \,\varepsilon \Rightarrow \,
	\sup_{t \geq 0} \| u(t) \|_{E} < \delta $$
	In addition, we say that it is asymptotically stable if  there exists $\eta > 0$
	such that $\lim_{t\rightarrow\infty} \|  u(t) \|_{E} = 0$ for all $u_0\in H_0^1(\Omega),\,
	\| u_0 \|_{E} < \eta. $.
\end{defi}

First, we have the following result:

 \begin{thm}\label{thstabil}
 Concerning the Dirichlet problem, assume the hypothesis of  Theorem \ref{existence} and $0 < \sigma_1 < \sigma_2$.

\begin{enumerate}[1.]
\medskip

\item $ L^p$  stability:
\medskip
 
 \noindent If 
 $$ k \leq 0, \,\,\, |\alpha| \frac{p-2}{2} \leq a, \,\,\, b\frac{\sigma_1}{\sigma_2} \leq c\,\,\, {\it and }\,\,\,
b \frac{ \sigma_2 - \sigma_1}{\sigma_2} \leq |k|$$
  the equilibrium point $0$ is $L^p$-stable for $2 \leq p \leq \frac {2 N} {N - 2}$, if $N > 2,\,\, 2 \leq p < \infty$ if $N=1, 2$.
 
 \noindent In addition, if $k < 0$ and $ b \frac{ \sigma_2 - \sigma_1}{\sigma_2} < |k|$ we have the
 asymptotic stability and
 $$\| u(t, x) \|_{L^p} \rightarrow 0\mbox{ as }t\rightarrow \infty, \quad \hbox{\it for all }  u_0\in H_0^1(\Omega).$$ 
 
 \noindent In the particular case $p=2$, if $\Omega$ is a bounded domain, 
 $$k > 0, \,\,\,  b\frac{\sigma_1}{\sigma_2} \leq c\,\,\, {\it and}\,\,\,
  b^+ \frac{ \sigma_2 - \sigma_1}{\sigma_2} + k <   a  \left(\frac{1} {\omega_N} | \Omega | \right)^{- 2/N}$$
   where $b^+ = \max \{0, b\}$,  $\omega_N$ represents
 the volume of the unit ball in $\R^N$ and $|\Omega |$ the volume of $\Omega$, then 
 $ \| u(t) \|_{L^2} \rightarrow 0$ as $t\rightarrow \infty $, for all $ u_0\in H_0^1(\Omega).$
 \medskip
 
 \item $H^1$ stability:
 \medskip
 
 \noindent  Assume $\alpha /a = \beta / b = \gamma / c$.  
 Then,  the equilibrium point $0$ is  asymptotically stable in $H^1$ if
 
 \smallskip
 \noindent 1. $ k < 0$ and
$$ \frac{b}{\sigma_1+2} \frac{\sigma_1}{\sigma_2} \leq \frac{c}{\sigma_2+2},\quad
 b \frac{ \sigma_2 - \sigma_1}{\sigma_2} \leq \frac{|k|}{2},\quad
 b  (\sigma_1 +1) < \min \{c, |k| \} $$
\smallskip
 
 \noindent 2. $k = 0,\, \Omega$ is a bounded domain and
 $$  \frac{b}{\sigma_1+2} \frac{\sigma_1}{\sigma_2} \leq \frac{c}{\sigma_2+2},\quad
  b (\sigma_1+1) < \min \left\{ c, 
 \left( \frac{1}{\omega_N} |\Omega| \right)^{-2/N} \right\}  $$
 In both cases,
 $$\| u(t) \|_{H^1} \rightarrow 0\,\mbox{ as } t\rightarrow \infty, \quad \hbox{\it for all} \ u_0\in H_0^1(\Omega).$$
 
 \end{enumerate}
 
 \end{thm}
 
 \begin{rem} If $b=0$, one may easily prove the asymptotic
 stability in $H^1$ with the additional condition
  $$ 0 <  k \leq  \frac {a} {2}  \left(\frac{1} {\omega_N} | \Omega | \right)^{- 2/N} $$
\end{rem}
 
 \begin{rem} The results stated in the theorem extended trivially, with a slight modification, to the 
 \eqref{gCGL} equation with a Neumann condition, in the case $k < 0$. 
 \end{rem}

Finally, we study the stability of some particular time periodic solutions of the generalized complex
Ginzburg-Landau equation. Consider the \eqref{gCGL} equation on a bounded domain $\Omega$
with the Neumann condition on the boundary and assume $0 < \sigma_1 < \sigma_2$.
Take the associated ordinary differential equation,
\begin{equation}\label{ord-GL}
\dot{u} = (b + i \beta) | u |^{\sigma_1} u - (c + i \gamma) | u |^{\sigma_2} u + ku 
\end{equation}
and look for periodic solutions. If we assume that there exists $r_0 > 0$ such that
\begin{equation}\label{condition}
b r_0^{\sigma_1} - c r_0^{\sigma_2} + k = 0
\end{equation}
we obtain the explicit periodic solution
\begin{equation}\label{periodic}
u(t) = r_0 {\rm exp} ( i t (\beta r_0^{\sigma_1} - \gamma r_0^{\sigma_2})) 
\end{equation}
We consider now the two following cases \medskip

\begin{itemize}
	\item[1.]  $c = 0$ 
	and  $b k < 0$; the condition \eqref{condition} can be verified and the equation \eqref{ord-GL}
	allows a $T_1 $-periodic solution \eqref{periodic} which we denote
	by $p(t),\,  T_1 > 0$. 
	
	\item[2.] $k = 0$ and $b c > 0$; we obtain a $T_2 $-periodic solution \eqref{periodic} which we denote
	by $q(t),\,  T_2 > 0$.
	\smallskip
	
\end{itemize}
\smallskip
It is clear that the \eqref{gCGL} equation with the Neumann condition on the boundary 
allows the time periodic solutions $P(x, t) \equiv p(t),\, Q(x, t) \equiv q(t)$ for all $x\in \Omega$.

\begin{thm} Let $\Omega \subset \R^N$ a bounded domain and consider the \eqref{gCGL} equation
	with a Neumann condition on the boundary. Suppose the conditions of Theorem \ref{existence} are verified. \smallskip
	
	\noindent 1. Assume $c=0$.\smallskip
	
	If $b<0$ and $k > 0$, the $T_1$-periodic solution $P(x, t)$ is orbitally asympto-\break
	tically stable, i.e. there exists $\delta > 0$ and $\zeta > 0$ such that, if
	$$ \min_{0\leq t\leq T_1} \| u_0 - P(t) \|_{H^1} < \delta $$
	the solution $u(t)$ of \eqref{gCGL} with initial data $u(0) = u_0$ exists on $0\leq t < \infty$ and there exists
	a real $\omega$ and $c > 0$ such that
	$$ \| u(t) - P(t - \omega) \|_{H^1} \leq c \,\delta\,  e^{- \zeta t} $$
	
	If $b > 0$ and $k < 0$, $P(x, t)$ is strongly unstable: for $u_n^0\equiv r_0+1/n$, the solution $u_n$ with initial condition $u_n(0)=u_n^0$ blows-up in finite time.
	\smallskip
	
	\noindent 2. Assume $k=0$. 
	\smallskip
	
	If $ b > 0$ and $c > 0$, the $T_2$-periodic solution $Q(x, t)$ is orbitally asymptotically stable.
	\smallskip
	
	If $b < 0$ and $ c < 0$, $Q(x, t)$ is strongly unstable.
	
\end{thm}
\begin{rem}
	In the Neumann case, the solutions of \eqref{ord-GL} automatically embed in the flow for \eqref{gCGL}. The above theorem says that, concerning the stability of $P$ and $Q$, both flows have precisely the same dynamic behavior.
\end{rem}

 \vskip10pt
 The paper is organized as follows: in Section 2, we prove the global existence result (Theorem 1.2). In Section 3, we focus on the construction of bound-states on the real line and on bounded domains. In Section 4, we study the stability of the trivial solution. Finally, Section 5 is devoted to the stability of periodic solutions.
 
\section{Proof of the Theorem 1.2}

\begin{proof}

To prove the global existence of a solution, multiply
 \eqref{gCGL} equation by $\overline {u}, \, -\Delta \overline {u}$ and
$|u|^\sigma_2 \overline{u}$, integrate on $\Omega$ and take the real part. One obtains

 \begin{equation} \label{conserv1}
      \frac{1} {2} \frac{d} {dt} \| u \|_{L^2}^2 = - a \| \nabla u \|_{L^2}^2
 + b \| u \|_{L^{\sigma_1+2}}^{\sigma_1 + 2} - c  \| u \|_{L^{\sigma_2+2}}^{\sigma_2 + 2} + k \| u \|_{L^2} 
 \end{equation}
 
 \begin{multline} \label{conserv2}
 \begin{split}
   \frac{1} {2} \frac{d} {dt} \| \nabla u \|_{L^2}^2  =  -a \| \Delta u \|_{L^2}^2 
   - b\, \Re \int_\Omega \Delta \overline{u} | u |^{\sigma_1} u \,dx + 
   \beta\, \Im \int_\Omega \Delta \overline{u} | u |^{\sigma_1} u \,dx \\
   + c\, \Re \int_\Omega \Delta \overline{u} | u |^{\sigma_2} u \,dx 
    - \gamma\, \Im \int_\Omega \Delta \overline{u} | u |^{\sigma_2} u \,dx + k \| \nabla u \|_{L^2}^2
   \end{split}
   \end{multline}
   \smallskip

    \begin{equation} \label{conserv3}
 \begin{split}
\frac{1} {\sigma_2 +2} \frac{d} {dt} \| u \|_{L^{\sigma_2+2}}^{\sigma_2 + 2} = \, &
a \,\Re  \int_\Omega \Delta u | u |^{\sigma_2} \overline{u} \,dx  
- \alpha \, \Im \int_\Omega \Delta u | u |^{\sigma_2} \overline{u} \,dx  \\
& + b \| u \|_{L^{\sigma_1 + \sigma_2+2}}^{\sigma_1 + \sigma_2 + 2} - c  \| u \|_{L^{2 \sigma_2+2}}^{2 \sigma_2 + 2}
+ k  \| u \|_{L^{\sigma_2+2}}^{\sigma_2 + 2}
\end{split}
\end{equation} \smallskip

\noindent Next, if we multiply \eqref{conserv3} by $\gamma / \alpha$ (with $\alpha \neq 0$) and add to
\eqref{conserv1} $+$ \eqref{conserv2}, one obtains:

\begin{multline}\label{estim0}
\frac{d}{dt} \left[ \frac {1} {2} \| u \|_{H^1}^2  + \frac {\gamma} {\alpha (\sigma_2 + 2)}
 \| u \|_{L^{\sigma_2+2}}^{\sigma_2 + 2} \right]
= k \left[  \| u \|_{H^1}^2 + \frac {\gamma}{\alpha} \| u \|_{L^{\sigma_2+2}}^{\sigma_2 + 2} \right]   
-a \| \Delta u \|_{L^2}^2  \\
- a \| \nabla u \|_{L^2}^2 + b \| u \|_{L^{\sigma_1+2}}^{\sigma_1 + 2} 
 -c \| u \|_{L^{\sigma_2+2}}^{\sigma_2 + 2} + 
 \frac {\gamma b} {\alpha} \| u \|_{L^{\sigma_1 + \sigma_2+2}}^{\sigma_1+ \sigma_2 + 2} 
  - \frac {\gamma c} {\alpha} \| u \|_{L^{2 \sigma_2+2}}^{2 \sigma_2 + 2} \\
 + \left( c + \frac{\gamma a} {\alpha}\right) \Re  \int_\Omega \Delta u | u |^{\sigma_2} \overline{u} \,dx 
 - b \,\Re  \int_\Omega \Delta \overline{u} | u |^{\sigma_1} u \,dx +
 \beta \,\Im \int_\Omega \Delta \overline{u} | u |^{\sigma_1} u \,dx 
\end{multline}
By interpolation we have
$$
\| u \|_{L^{\sigma_1+2}}^{\sigma_1+2} \leq \| u \|_{L^2}^\frac{2 (\sigma_2 - \sigma_1)}{\sigma_2}
\| u \|_{L^{\sigma_2+2}}^\frac{\sigma_1 (\sigma_2+2)}{\sigma_2}
$$

\noindent and by the well-known Young inequality
$$ ab \leq \varepsilon a^p + \frac{p-1}{p^{p'}} \varepsilon^\frac{1}{1-p} b^{p'},\quad p > 1,\,\, \varepsilon > 0, $$
with $p = \sigma_2/\sigma_1$, we obtain
\begin{equation*}
\| u \|_{L^{\sigma_1+2}}^{\sigma_1+2} \leq \varepsilon \| u \|_{L^{\sigma_2+2}}^{\sigma_2+2} +
\frac{p-1}{p^{p'}} \varepsilon^\frac{1}{1-p} \| u \|_{L^2}^2
\end{equation*}
and we choose $\varepsilon$ such that $ b  \varepsilon = c$ (if $b > 0$). It follows that
\begin{equation}\label{estim1}
b \| u \|_{L^{\sigma_1+2}}^{\sigma_1+2} - c \| u \|_{L^{\sigma_2+2}}^{\sigma_2+2}  < c_1 \| u \|_{L^2}^2
\end{equation}
with $ c_1 = | b | \frac{p-1}{p^{p'}} \varepsilon^\frac{1}{1-p} $.
 Similarly 
$$ \| u \|_{L^{\sigma_1+ \sigma_2 +2}}^{\sigma_1+ \sigma_2 +2} 
\leq \delta \| u \|_{L^{2 \sigma_2+2}}^{2 \sigma_2+2} + C(\delta) \| u \|_{L^2} 
$$

\noindent and if we choose $\delta$ such that $\delta b  < c/2$ (if $b > 0$), we get
\begin{equation}\label{estim2}
\frac{\gamma b}{\alpha} \| u \|_{L^{\sigma_1+ \sigma_2 +2}}^{\sigma_1+ \sigma_2 +2}  -
\frac{\gamma c}{ \alpha} \| u \|_{L^{2 \sigma_2 +2}}^{2 \sigma_2 +2}  
< - \frac{\gamma c}{2 \alpha} \| u \|_{L^{2 \sigma_2 +2}}^{2 \sigma_2 +2}  + c_2  \| u \|_{L^2}^2
\end{equation}
with $c_2 = c_2(\delta)$. Next, we estimate

\begin{equation}\label{estim3}
\begin{split}
(| b | + |\beta|) \left|\int_\Omega \Delta \overline{u} |u|^{\sigma_1} u \right| \leq
(| b | + |\beta|) \| \Delta \|_{L^2} \| u \|_{2 \sigma_1 +2}^{\sigma_1 +1} \\
\leq \eta (| b | + |\beta|) \| \Delta \|_{L^2}^2  + \frac{|b|+|\beta|}{\eta}  \| u \|_{2 \sigma_1 +2}^{2\sigma_1 +2}
\end{split}
\end{equation}
and we take $\eta$ such that $\eta (|b| + |\beta|) \leq a$. By interpolation,
$$ \| u \|_{2 \sigma_1+2} \leq \| u \|_{2 \sigma_2+2}^\frac{2 \sigma_1(\sigma_2+1)}{\sigma_2}
\| u \|_{L^2}^\frac{2 (\sigma_2 - \sigma_1)}{\sigma_2} $$
and using the Young inequality (with $p= \frac{\sigma_2}{\sigma_2 - \sigma_1},\, p' = \frac{\sigma_2}{\sigma_1}$)
we get
\begin{equation}\label{estim4}
 \| u \|_{2 \sigma_1+2}^{2 \sigma_1+2} \leq  \rho \| u \|_{2 \sigma_2+2}^{2 \sigma_2 + 2}
+ C(\rho) \| u \|_{L^2}^2
\end{equation}
and we choose $\rho$ such that $\frac{| b | + | \beta |}{\eta} \rho \leq \frac{\gamma c}{\alpha}$.
Finally, notice that $ \Re  \int_\Omega \Delta \overline{u} | u |^{\sigma_2} u \,dx  \leq 0$. By \eqref{estim1},
\eqref{estim2}, \eqref{estim3}, \eqref{estim4} and \eqref{estim0} we obtain the conclusion by the
Gronwall inequality. 
\end{proof} \smallskip

\begin{rem} The complex Ginzburg-Landau equation on $\Omega \subset \R^N$ with the Dirichlet condition,
$$ u_t = e^{i\theta} \left( \Delta u + | u |^\sigma u\right) + ku,\quad k\geq 0, \,\,-\pi/2 < \theta < \pi/2 $$
allows explosive solutions in a finite time, $u(t)$, under the condition that the energy
$$   \frac{1}{2} \int_\Omega | \nabla u_0 |^2 dx - \frac{1}{\sigma+2}  \int_\Omega| u_0|^{\sigma+2}  dx < 0 $$
(see \cite{Caz, Caz-Weiss1}). This result remains true (with essentially the same proof) for the generalized Ginzburg-Landau equation:
\begin{equation}\label{gGL*}
u_t = e^{i\theta} \left[ \Delta u + | u |^{\sigma_1} u - \nu | u |^{\sigma_2} u \right] + ku,\quad k\geq 0,\,\,\nu\in\R
\end{equation}
More precisely, we have
\begin{prop} Assume $-\pi/2 < \theta < \pi/2, \,\sigma_1, \sigma_2 > 0,\,\, k \geq 0$ and
$\nu \leq 0$ or  $\nu > 0$ and $\sigma_2 \leq \sigma_1$.
Let $u_0\in H_0^1(\Omega)$ and $u(t)$ the corresponding maximal solution of \eqref{gGL*}. If 
$ E(u_0) < 0$ with
$$ E(u_0) =  \int_\Omega (\frac{1}{2}  | \nabla u_0 |^2  - \frac{1}{\sigma_1+2}  | u_0|^{\sigma_1+2} 
+ \nu  \frac{1}{\sigma_2+2}  | u_0|^{\sigma_2+2} ) dx < 0 $$
then $u$ blows up in a finite time.
\end{prop}

\end{rem}

\vskip10pt
\section{Existence of bound-states of \eqref{gGLcomplex}}
\vskip 10pt

\begin{proof}[Proof of Theorem \ref{thmbs}]  We look for solutions $\phi\in H^1(\R)$ of the elliptic
equation \eqref{gGLelip1}, or in an equivalent form,
\begin{equation}\label{gGLelip2}
\phi'' = \omega e^{i \tilde{\theta}} \phi - e^{i \tilde{\gamma_1}} | \phi |^{\sigma_1} \phi
+ \chi  e^{i \tilde{\gamma_2}} | \phi |^{\sigma_2} \phi + i k e^{i \tilde{\theta}} \phi 
\end{equation}
with $\tilde{\theta} = \pi/2 - \theta,\, \tilde{\gamma_1} = \gamma_1 - \theta, \, \tilde{\gamma_2} = \gamma_2 - \theta$.

\smallskip

\noindent 1. First consider the case $\chi = 1$.

\smallskip

Let us search for a solution $\phi\in H^1 (\R)$ of the equation \eqref{gGLelip2} of the form
$$ \phi = \psi \exp (i d \ln \psi) $$
where $d\in \R$ and $\psi > 0$ is the unique solution (up to translations of the origin) of the 
stationary Schr\"odinger equation
\begin{equation}\label{est-schr}
\psi'' = \varepsilon \psi - \eta \psi^{\sigma_1+1} - \zeta \psi^{\sigma_2+1} =: - f(\psi),\quad 
\varepsilon, \eta, \zeta > 0 
\end{equation}
Note that the existence of the solution $\psi$ follows from the fact that
$$ x_0 = \inf \{ x   > 0 : F(x) =0 \} > 0 \quad{\rm with} \quad F(z) = \int_0^z f(s) ds $$
and $f(x_0 ) > 0$ (see \cite{P-L}, Th.5).

\noindent First, one has
 	\begin{equation}\label{second-deriv}
 	\phi'' ( x ) = \left[ \psi'' (x) (1 + i d) + id (1 + i d) \frac{\psi'(x)^2} {\psi(x)} \right] \exp ( i d \ln \psi(x) ).
 	\end{equation}
  and we note that if $\psi$ is a solution of \eqref{est-schr}, then a direct integration of the equation yields
 	\begin{equation}\label{Schr-new}
 	\frac{(\psi')^2} { \psi} = \epsilon\, \psi - \frac {2 \eta} {\sigma_1 + 2} \,\psi^{\sigma_1 + 1} -
	\frac {2 \zeta} {\sigma_2 + 2} \,\psi^{\sigma_2 + 1}.
 	\end{equation}
 	It follows from \eqref{gGLelip2} that
 	\begin{align*}
 		\psi'' - d^2 \frac{(\psi')^2} {\psi} & = \omega\cos \tilde{\theta} \,\psi - k \sin \tilde{\theta} \,\psi - \cos \tilde{\gamma_1} \,\psi^{\sigma_1 +1} -  \cos \tilde{\gamma_2} \, \psi^{\sigma_2 +1}, \\
 		d \psi'' + d \frac{(\psi')^2} {\psi} & = \omega\sin \tilde{\theta} \, \psi + k \cos \tilde{\theta} \,\psi - \sin \tilde{\gamma_1} \,\psi^{\sigma_1 +1} - \sin \tilde{\gamma_2} \,\psi^{\sigma_2 +1}
 	\end{align*}
 	and so
 	
 	\begin{multline*}
 		( 1 + d^2) \psi''   = [\omega(d \sin \tilde{\theta} + \cos \tilde{\theta}) 
 		+ k (  d \cos \tilde{\theta} - \sin \tilde{\theta})] \,\psi  \\
 		- ( d \sin \tilde{\gamma_1} + \cos \tilde{\gamma_1} ) \,\psi^{\sigma_1 +1} - 
		( d \sin \tilde{\gamma_2} + \cos \tilde{\gamma_2} ) \,\psi^{\sigma_2 +1},
 	\end{multline*}
 	
 	\begin{multline*}
 		(1 + d^2) \frac{(\psi')^2} {\psi}  = \biggl[ \omega\biggl(\frac{\sin \tilde{\theta}} {d} - \cos \tilde{\theta}\biggl) \psi
 		+ k  \biggl(\frac{\cos \tilde{\theta}} {d} + \sin \tilde{\theta}\biggl) \psi \\
 		- \biggl( \frac{\sin \tilde{\gamma_1}} {d} - \cos \tilde{\gamma_1} \biggl) \psi^{\sigma_1 +1}
		 - \biggl( \frac{\sin \tilde{\gamma_2}} {d} - \cos \tilde{\gamma_2} \biggl) \psi^{\sigma_2 +1} \biggl] .
 	\end{multline*}
 	Hence, writing
 	\begin{equation*}\label{alfa}
 	\epsilon = \frac { \omega(d \sin \tilde{\theta} + \cos \tilde{\theta}) + k ( d \cos \tilde{\theta}  - \sin \tilde{\theta}) }
 	{ 1 + d^2}
 	\end{equation*}
 	
 	\begin{equation*}\label{beta}
 	\eta = \frac { d \sin \tilde{\gamma_1} + \cos \tilde{\gamma_1} }
 	{ 1 + d^2}
	\end{equation*}
	 and
	 \begin{equation*}\label{gama}
	  \zeta = \frac { d \sin \tilde{\gamma_2} + \cos \tilde{\gamma_2} }{1 + d^2}
 	\end{equation*}
 	we require that
 	\begin{equation}\label{equ-final1}
 	\omega\biggl( \frac{\sin \tilde{\theta}} {d} - \cos \tilde{\theta}\biggl) + k \biggl( \frac{\cos \tilde{\theta}} {d} + \sin \tilde{\theta} \biggl) 
 	= \omega(d \sin \tilde{\theta} + \cos \tilde{\theta} )
 	+ k ( d \cos \tilde{\theta} - \sin \tilde{\theta})
 	\end{equation}

 	\begin{equation}\label{equ-final2}
 	\frac{\sin \tilde{\gamma_1}} {d} - \cos \tilde{\gamma_1} = \frac{2}  {\sigma_1 + 2}
 	( d \sin \tilde{\gamma_1} + \cos \tilde{\gamma_1} ) 
 	\end{equation}
	and
	\begin{equation}\label{equ-final3}
 	\frac{\sin \tilde{\gamma_2}} {d} - \cos \tilde{\gamma_2} = \frac{2}  {\sigma_2 + 2}
 	( d \sin \tilde{\gamma_2} + \cos \tilde{\gamma_2} ) 
 	\end{equation}

 	\noindent From \eqref{equ-final1} we derive
 	\begin{equation}\label{valor-d}
 	d = \frac{ k \sin \tilde{\theta} - \omega\cos \tilde{\theta}  \pm \sqrt{\omega^2 + k^2}}
 	{\omega\sin \tilde{\theta} + k \cos \tilde{\theta}} =: d_{\pm}
 	\end{equation}
 	and so
 	$$
 	\epsilon=\pm\sqrt{\omega^2+k^2}.
 	$$
 	Since $\epsilon > 0$ (see \cite{P-L}), we must have $d = d_+$. Finally, the conditions \eqref{equ-final2},
	\eqref{equ-final3} and $\eta, \zeta > 0$ are equivalents to \eqref{eq:restricaogamma}.
	
	\smallskip
	
	\noindent 2. Now we consider the case $\chi = -1$.
	\smallskip
	
	Keeping the same notation, we obtain again the conclusions \eqref{equ-final1}, \eqref{equ-final2},
	and \eqref{equ-final3} assuming the existence of the solution of the stationary Schr\"odinger
	equation
	$$ \psi'' = \epsilon \psi - \eta \psi^{\sigma_1+1} + \zeta \psi^{\sigma_1+1}  $$
	with $\epsilon, \eta, \zeta > 0$. Set $ f(z) = - \epsilon z + \eta z^{\sigma_1+1} - \zeta z^{\sigma_2+1}$
	and take the primitive
	$$ F(z) = \int_0^z f(s) ds = \frac{z^2} {2} \left[ - \epsilon + \frac{2 \eta} {\sigma_1+2} z^{\sigma_1}
	- \frac {2 \zeta} {\sigma_2+2} z^{\sigma_2} \right] $$
	It is clear that $z_0 := \inf \{ z > 0 : F(z) = 0 \} > 0 $ and
	$$ f(z_0) = z_0 \left[ - \epsilon + \eta z_0^{\sigma_1} - \zeta z_0^{\sigma_2} \right] > 0 $$
	since
	
	$$ - \epsilon + \eta z_0^{\sigma_1} - \zeta z_0^{\sigma_2} = $$
	$$ - \epsilon + \frac{2 \eta} {\sigma_1+2} z_0^{\sigma_1} + \left(\eta - 
	\frac{2 \eta} {\sigma_1+2} \right) z_0^{\sigma_1}
	 - \frac{2 \zeta} {\sigma_2+2} z_0^{\sigma_2} + \left(  \frac{2 \zeta} {\sigma_2+2} - \zeta\right) z_0^{\sigma_2} = $$
	 
	$$  z_0^{\sigma_2} \left[ \frac{\sigma_1 \eta} {\sigma_1 +2} z_0^{\sigma_1 - \sigma_2}
	 - \frac{\sigma_2 \zeta} {\sigma_2 +2} \right] > 0 $$ 
	 \smallskip
	 
	\noindent  which is verified for $\sigma_2 $ small enough.
\end{proof}
\vskip 10pt
The remainder of this section is dedicated to the proof of Theorem \ref{teo:bifurca}. Throughout the proof, $(\cdot, \cdot)$ will denote the complex $L^2$-inner product
$$
(u,v)=\int_\Omega u(x)\overline{v(x)}dx,\quad u,v\in L^2(\Omega).
$$

Denote by $\lambda_0$ a double eigenvalue of the Laplace-Dirichlet operator 
$- \Delta$ in $L^2 (\Omega)$ and let
$u_1, u_2\in H^\infty (\Omega)\cap H^1_0(\Omega)$ be two $L^2$-orthonormal eigenfunctions, spanning the eigenspace $V$. Furthermore, define the orthogonal projection $P:L^2(\Omega)\to V^\perp$. As a consequence, for all $\lambda$ near $\lambda_0$, one has
$$
\| (\lambda + P\Delta)^{-1} f \|_{H^1_0(\Omega)} \lesssim \| f  \|_{H^{-1}(\Omega)},\quad \mbox{for all }f\in H^{-1}(\Omega). $$
%
To simplify notations, set
$$
L:= - e^{i\theta} \Delta,\quad M_1u=e^{i\gamma_1}|u|^{\sigma_1}u , \quad M_2u= \chi e^{i\gamma_2}|u|^{\sigma_2}u,\quad Mu=M_1u + M_2u.
$$
Then equation \eqref{gGLelip1} can be rewritten as
\begin{equation} \label{eigen-equ}
\lambda u - L u +Mu = 0, \quad \lambda=k-i\omega \in \C.
\end{equation}
Applying the Lyapunov-Schmidt reduction, equation \eqref{eigen-equ} is equivalent to the system
\begin{equation}\label{equ-1}
P (\lambda u - L u + Mu) = 0
\end{equation}
\begin{equation}\label{equ-2}
( \lambda u - L u + Mu, u_j) = 0,\quad j=1,2.
\end{equation}
We write $u=y+\epsilon_1u_1+\epsilon_2u_2$, $y\in V^\perp$. Since \eqref{eigen-equ} enjoys a gauge symmetry, we may assume, without loss of generality, that $\epsilon_1>0$. By \eqref{equ-1}, 
\begin{equation}\label{newequ-1}
y = (\lambda - P L)^{-1} [ - P M (y + \epsilon_1u_1+\epsilon_2 u_2) ] 
\end{equation}
On the other hand, equation \eqref{equ-2} reduces to
\begin{equation}\label{eq:lambda_interm}
\epsilon_j(\lambda-e^{i\theta}\lambda_0 ) = -(M(y+\epsilon_1u_1+\epsilon_2u_2), u_j),\quad j=1,2.
\end{equation}
Setting $\alpha=\epsilon_2/\epsilon_1$, it follows from \eqref{eq:lambda_interm} that
\begin{align}\label{eq:lambda}
\lambda-e^{i\theta}\lambda_0   &= -\frac{1}{\epsilon_1}(M(y+\epsilon_1u_1+\epsilon_2u_2), u_1)\nonumber\\&= -\epsilon_1^{\sigma_1}(M_1(u_1+\alpha u_2), u_1) + Q_1(y,\epsilon_1,\alpha),
\end{align}
where
$$
Q_1(y,\epsilon_1,\alpha) = -\frac{1}{\epsilon_1}(M(y+\epsilon_1u_1+\epsilon_2u_2)-M_1(\epsilon_1u_1+\epsilon_2u_2), u_1).
$$
On the other hand, again by \eqref{eq:lambda_interm},
\begin{align}\label{eq:interm_alpha}
\alpha(M(y+\epsilon_1u_1 + \epsilon_2u_2), u_1) &= -\alpha \epsilon_1(\lambda-e^{i\theta}\lambda_0  ) = -\epsilon_2(\lambda-e^{i\theta}\lambda_0  ) \nonumber\\&= (M(y+\epsilon_1u_1 + \epsilon_2u_2), u_2).
\end{align}
Setting
$$
Q_2(y,\epsilon_1,\alpha)=\frac{1}{\epsilon_1^{\sigma_1+1}}(M(y+\epsilon_1u_1 + \epsilon_2u_2)-M_1(\epsilon_1u_1 + \epsilon_2u_2), \overline{\alpha} u_1-u_2),
$$
equation \eqref{eq:interm_alpha} becomes
\begin{equation}\label{eq:alpha}
(M_1(u_1 + \alpha u_2), \overline{\alpha} u_1 - u_2) +  Q_2(y,\epsilon_1,\alpha)=0.
\end{equation}
%
%
%
The proof of Theorem \ref{teo:bifurca} will follow from the following steps: first, we show that, for each $\epsilon_1, \alpha, \lambda$ fixed, $y$ may be found through a fixed-point argument applied to \eqref{newequ-1}. Afterwards, we apply a Lipschitz version of the Implicit Function Theorem to solve \eqref{eq:lambda} and \eqref{eq:alpha}.
Since $\lambda_0  $ is not an eigenvalue of $P\Delta$, there exists  $r > 0$ small such that
\begin{equation}\label{V-r}
V_r := \{ \lambda \in \C : | \lambda - e^{i\theta}\lambda_0   | \leq r \} \subset \rho (PL) 
\end{equation}
Notice that
$$
H^1_0(\Omega)\hookrightarrow L^{\sigma_j+2}(\Omega),\quad j=1,2,
$$
and, by duality,
$$
 L^{\frac{\sigma_j+2}{\sigma_j+1}}(\Omega)\hookrightarrow H^{-1}(\Omega),\quad j=1,2.
$$
\begin{lem}\label{lem:construcao_y}
	Let $ \lambda \in V_r$. Then, for all $\delta > 0$ small enough and $ |\epsilon_1|, |\epsilon_2| \leq \delta$, there exists a solution 
	$ y = y(\epsilon_1,\epsilon_2,\lambda) \in H^1_0(\Omega)$
	of \eqref{newequ-1}. Moreover, for some universal constants $C > 0, K > 0$,
	\begin{equation}\label{estim-y}
	\| y(\epsilon_1,\epsilon_2,\lambda)\|_{H^1_0} \leq C \delta^{{\sigma_1}+1}  \\
	\end{equation}
	and
	\begin{equation}\label{estim-lip}
	\|  y(\epsilon_1,\epsilon_2,\lambda) -  y(\tilde{\epsilon}_1,\tilde{\epsilon}_2,\tilde{\lambda}) \|_{H^1_0} \leq K\delta^{{\sigma_1+1}}| (\epsilon_1,\epsilon_2,\lambda) - (\tilde{\epsilon}_1,\tilde{\epsilon}_2,\tilde{\lambda}) | \\
	\end{equation}
	for all  $\lambda_1, \tilde{\lambda_1} \in V_r $ and $|\epsilon_j|, |\tilde{\epsilon}_j|<\delta$, $j=1,2$.
	%
	%
\end{lem}

\begin{proof} Denote by $S y= S(\epsilon_1,\epsilon_2,\lambda)y$ the right-hand side of \eqref{newequ-1} and
	by $R(\lambda)$ the resolvent $ (\lambda - PL)^{-1}$, which is a bounded operator from $H^{-1}(\Omega)$ to $H^1_0(\Omega)$, uniformly in $\lambda\in V_r$. Then, for 
	$\| y_1 \|_{H^1_0}, \| y_2\|_{H^1_0} \leq \delta$ and $r_j=(\sigma_j+2)/(\sigma_j+1),\ j=1,2$,
		\begin{align}
		\| S y_1\|_{H^1_0}  \lesssim \ &
	\| M(y_1 + \epsilon_1u_1 + \epsilon_2u_2) \|_{H^{-1}}\nonumber \\ \lesssim \ &
	\sum_{j=1,2}\| M_j(y_1 + \epsilon_1u_1 + \epsilon_2u_2)\|_{L^{r_j}} \label{eq:limitacao}\\ \lesssim \ &
	\sum_{j=1,2} \left(\|y_1\|_{L^{\sigma_j+2}} + \epsilon_1\|u_1\|_{L^{\sigma_j+2}} + |\epsilon_2|\|u_2\|_{L^{\sigma_j+2}} \right)^{\sigma_j+1} \lesssim (\delta^{\sigma_1+1} + \delta^{\sigma_2+1})\nonumber
	\end{align}
	and
	\begin{align}
	&	\| S y_1 - S y_2 \|_{H^1_0}  \lesssim 
	\| M(y_1 + \epsilon_1u_1 + \epsilon_2u_2) -  M(y_2 +  \epsilon_1u_1 + \epsilon_2u_2) \|_{H^{-1}}\nonumber \\ \lesssim \ &
	\sum_{j=1,2}\| M_j(y_1 + \epsilon_1u_1 + \epsilon_2u_2) -  M_j(y_2 +  \epsilon_1u_1 + \epsilon_2u_2) \|_{L^{r_j}} \nonumber\\ \lesssim \ &
\sum_{j=1,2} \left(\|y_1\|_{L^{\sigma_j+2}} + \|y_1\|_{L^{\sigma_j+2}} + \epsilon_1\|u_1\|_{L^{\sigma_j+2}} + |\epsilon_2|\|u_2\|_{L^{\sigma_j+2}} \right)^{\sigma_j}\| y_1-y_2\|_{L^{\sigma_j+2}} \nonumber\\\lesssim\ & (\delta^{\sigma_1} + \delta^{\sigma_2})\|y_1 - y_2\|_{H^1_0}\label{eq:contracao}
	\end{align}
	Thus, for all $\delta>0$ small, it follows from the Banach fixed point theorem that there exists a unique 
	$y = y (\epsilon_1, \epsilon_2, \lambda)$ solution to \eqref{newequ-1}
	with $ \| y(\epsilon_1, \epsilon_2, \lambda) \|_{H^1_0} \leq \delta$. By \eqref{eq:contracao}, this estimate can be improved to \eqref{estim-y}. 
	
	We now prove the Lipschitz estimate \eqref{estim-lip} in $\lambda$, as the estimate  in the remaining variables is straightforward.
	For $|\epsilon_1|, |\epsilon_2|<\delta$ fixed, take $\lambda_1, \lambda_2 \in V_r$ and consider $y_j=y(\epsilon_1, \epsilon_2, \lambda_j)$, $j=1,2$. Then
	\begin{equation*}
	\begin{split}
	y_1 -  y_2 = & ( R(\lambda_1) - R(\lambda_2) ) 
	[  P M (  y_1 + \epsilon_1u_1 + \epsilon_2u_2) ]\\
	& + R(\lambda_2) \big[  P M (  y_1 + \epsilon_1u_1 + \epsilon_2u_2) 
	-  P M ( y_2+ \epsilon_1u_1 + \epsilon_2u_2) \big ] \\
	= & ( \lambda_2 - \lambda_1)  R(\lambda_1)  R(\lambda_2) 	[ P M (  y_1 + \epsilon_1u_1 + \epsilon_2u_2) ]\\
	& + R(\lambda_2) \big[  P M (  y_1 + \epsilon_1u_1 + \epsilon_2u_2) 
	-  P M ( y_2+ \epsilon_1u_1 + \epsilon_2u_2) \big ].
	\end{split}
	\end{equation*}
	Therefore, proceeding as in \eqref{eq:contracao},
	\begin{align*}
	\|y_1-y_2\|_{H^1_0}&\lesssim \sum_{j=1,2}|\lambda_1-\lambda_2|\left(\|y_1\|_{H^1_0} + \epsilon_1 + |\epsilon_2|\right)^{\sigma_j+1} \\&\quad+ \left( \|y_1\|_{H^1_0} +\|y_2\|_{H^1_0}+ \epsilon_1 + |\epsilon_2|\right)^{\sigma_j}\|y_1-y_2\|_{H^1_0}
	\\&\lesssim (\delta^{\sigma_1+1} + \delta^{\sigma_2+1})|\lambda_1-\lambda_2| + (\delta^{\sigma_1} + \delta^{\sigma_2})|y_1-y_2\|_{H^1_0}
	\end{align*}
	The estimate follows for $\delta$ small enough.
\end{proof}

\begin{proof}[Proof of Theorem \ref{teo:bifurca}]
	We wish to solve system \eqref{eq:lambda}-\eqref{eq:alpha}. First, when one drops the remainder terms $R$ and $Q$, the system reduces to
	\begin{equation}\label{eq:sistemaimplicitasimplificado}
	\left\{\begin{array}{l}
	F_1(\epsilon_1, \lambda, \alpha)=\lambda-e^{i\theta}\lambda_0  +\epsilon_1^{\sigma_1}(M_1(u_1+\alpha u_2), u_1)=0\\
	F_2(\epsilon_1, \lambda, \alpha)=P(\alpha)=0
	\end{array}\right.
	\end{equation}
	which, by assumption, satisfies the conditions of the Implicit Function Theorem at $(\epsilon_1, \lambda, \alpha)=(0,e^{i\theta}\lambda_0  , \alpha_0)$. Now observe that, due to \eqref{estim-y} and \eqref{estim-lip}, $Q_1$ and $Q_2$ are Lipschitz continuous in $\epsilon_1, \lambda$ and $\alpha$, with constant proportional to $\epsilon_1^{\sigma_1-1}$, $\epsilon_1^{\sigma_1}$ and $\epsilon_1^{\sigma_1}$, respectively. We exemplify by proving the Lipschitz estimate for $Q_2$ with respect to $\lambda$: for $\epsilon_1$ small and $\alpha$ fixed, using H\"older inequality,
	\begin{align*}
	&\quad |Q_2(y(\lambda_1))-Q_2(y(\lambda_2))|\\&\lesssim \frac{1}{\epsilon_1^{\sigma_1+1}}|(M(y(\lambda_1)+\epsilon_1 u_1 + \epsilon_2 u_2)-M(y(\lambda_2)+\epsilon_1 u_1 + \epsilon_2 u_2),u_1)|\\&\lesssim \frac{1}{\epsilon_1^{\sigma_1+1}} \sum_{j=1,2}\left( \|y(\lambda_1)\|_{L^{\sigma_j+2}} + \|y(\lambda_2)\|_{L^{\sigma_j+2}} + \epsilon_1 + \epsilon_2 \right)^{\sigma_j+1}\|y(\lambda_1)-y(\lambda_2)\|_{L^{\sigma_j+2}} \\&\lesssim \frac{1}{\epsilon_1^{\sigma_1+1}}\left(\epsilon_1^{\sigma_1} + \epsilon_2^{\sigma_2}\right)\|y(\lambda_1)-y(\lambda_2)\|_{H^1_0}  \lesssim \epsilon_1^{\sigma_1}|\lambda_1-\lambda_2| 
	\end{align*}
	
	Therefore \eqref{eq:lambda}-\eqref{eq:alpha} is a Lipschitz perturbation, small in $\alpha$ and  $\lambda$, of \eqref{eq:sistemaimplicitasimplificado}. The conclusion now follows from \cite[Section 7.1]{Clarke}.
\end{proof}

\begin{rem}
	The above proof can be easily applied to the case of simple eigenvalues. Indeed, the Lyapunov-Schmidt reduction yields the system
	\begin{equation}\label{newequ-simples}
	y = (\lambda - P L)^{-1} [ - P M (y + \epsilon_1u_1) ] 
	\end{equation}
	\begin{equation}\label{eq:lambdasimples}
	\lambda-e^{i\theta}\lambda_0 = -\frac{1}{\epsilon_1}(M(y+\epsilon_1u_1), u_1)
	\end{equation}
	The first equation can be solved through a fixed point argument, while the second is in the conditions of the Implicit Function Theorem (in the Lipschitz formulation).
\end{rem}

\begin{exemp}\label{exemplo}
	Let $\Omega=(-1,1)^2$ and $\sigma_1=2$. As it is well-known, the second eigenvalue of the Laplacian $\lambda_2=5\pi^2/4$ is double, with associated eigenfunctions
	$$
	v_1(x,y)=\cos\left(\frac{\pi x}{2}\right)\sin\left(\pi y\right), \quad v_2(x,y)=\cos\left(\frac{\pi y}{2}\right)\sin\left(\pi x\right).
	$$
	If we choose $u_1=v_1$ and $u_2=v_2$, the function $P$ takes the form
	$$
	P(\alpha)=\frac{3}{16}(\alpha^3-\alpha).
	$$
	and we find three bifurcation branches with $\alpha_0=0, \pm 1$. The permutation $u_1=v_2$, $u_2=v_1$ provides yet another branch (formally identifiable with $\alpha_0=\infty$).
	In conclusion, we recover the results of \cite{DelPinoGarciaMusso}, which are specific for $\Omega=(-1,1)^2$, $\sigma_1=2$ and $ \theta, \gamma_1, \gamma_2=0$.
\end{exemp}

\section{Stability of the trivial equilibrium}

\vskip 10pt

In this section, we study the stability of the equilibrium solution $u\equiv 0$
and the  asymptotic decay of global solutions of \eqref{gCGL}
depending on the parameters and the coeficient for the driving term $k$. Let denote by $S(t)$ the dynamical system
associated to  \eqref{gCGL}:  $S(t) u_0 \equiv u(t; u_0),\, t \geq 0$. \smallskip

 \begin{defi}
We say that $u_0 \in H_0^1(\Omega)$ is stable if for any $\delta > 0$ there exists $\varepsilon > 0$ such that
$$ v\in H_0^1 (\Omega) , \,\, \| u_0 - v \|_{H^1} < \,\varepsilon \Rightarrow \,
\sup_{t \geq 0} \| S(t) u_0 - S(t) v \|_{H^1} < \delta $$
In addition, we say that $u_0$ is asymptotically stable if $u_0$ is stable and there exists $\eta > 0$
such that $\lim_{t\rightarrow\infty} \| S(t) u_0 - S(t) y \|_{H^1} = 0$ for all $y\in H_0^1(\Omega),\,
\| u_0 - y \|_{H^1} < \eta. $
\end{defi}

More generally if $S(t)$ denote a a dynamical system on a Banach space $H$ we recall that a Lyapunov
function is a continuous function $W : H \rightarrow \R$ such that
$$ \dot W (u) := \limsup_{t\rightarrow 0^+} \, \frac{1} {t} [ W (S(t) u) - W(u) ] \leq 0$$
for all $u\in H$. 
The next lemma is mainly proved in \cite{Hale}.

\begin{lem}\label{lemaHale}
Let $S(t)$ be a dynamical system on a Banach space $(D, \| \, \|)$. Let $E$ a normed space
such that $D\hookrightarrow E$ and $W$ a Lyapunov function on $D$ such that
$$ W(u_0) \geq k_1 \| u_0 \|_E,\,\, k_1 > 0,\, u_0 \in D. $$
Then, the equilibrium point $0$ is $\|\,\|_E$ - stable in the sense that
$$ u_0 \in D,\,\,  \| u_0 \|\rightarrow 0 \Rightarrow \| S(t) u_0 \|_E \rightarrow 0, $$
 uniformly in $t \geq 0$.

\noindent Assume in addition that
$$ \dot W (u_0) \leq - k_2 \| u_0 \|_E,\,\, k_2 > 0, \, u_0 \in D.$$
Then, $lim_{t\rightarrow \infty} \| S(t) u_0 \|_E = 0$ for any $u_0 \in D$.

\end{lem} 
\medskip
 
\begin{proof}[Proof of Theorem 1.6]
 1. Let us denote by $S(t) u_0 \equiv u(t, u_0)$ the unique global solution of \eqref{gCGL} 
under the hypothesis of the Theorem \eqref{existence} and define
$$ W_p (u) = \int_\Omega | u(x) |^p dx,$$
with $ 2 \leq p \leq \frac{2 N} {N - 2} $ if $N > 2$, $ 2 \leq p < \infty$ if $N =1, 2$
and $ u = u(t, u_0)$. It is clear that $W_p : H_0^1 (\Omega) \rightarrow \R$, is a continuous functional and,
from  $\dot W_p (u) = \nabla W(u) \cdot \frac {d} {d t} u(t)$, we get

\begin{multline*}
\dot W_p(u)  = \, p \,\Re \int_\Omega | u |^{p-2} \overline{u} \{ (a + i \alpha) \Delta u +
 (b + i \beta) | u |^{\sigma_1} u  \\
  -  (c + i \gamma) | u |^{\sigma_2} u + ku \} \,dx 
  \end{multline*}
  \begin{align}
  \le & \,p k \int_\Omega | u |^p dx - ap \int_\Omega | u |^{p-2} | \nabla u|^2 dx 
  + p b \int_\Omega | u |^p | u |^{\sigma_1} dx \\
  - & p c \int_\Omega | u |^p | u |^{\sigma_2} dx + 
   p  \alpha \,\Im \int_\Omega \nabla \left(| u |^{p-2}\right) \overline{u} \nabla u\, dx.  \notag
\end{align}

\noindent Since
$$ \nabla | u |^{p-2} = \frac{p-2} {2} | u |^{p-4} ( u \nabla \overline{u} + \overline{u} \nabla u) $$
we obtain
$$ \left| p \alpha \Im \int_\Omega \nabla \left( | u |^{p-2}\right) \overline {u} \nabla u dx \right| \leq
p | \alpha | \frac {p-2} {2} \int_\Omega | u |^{p-2} | \nabla u |^2 dx. $$
Furthermore, by interpolation, one has
$$ \| u \|_{L^{p+\sigma_1}}^{p+\sigma_1} \leq 
\| u \|_{L^{p+\sigma_2}}^\frac{(p+\sigma_2) \sigma_1}{\sigma_2 }
 \| u \|_{L^p}^\frac{p (\sigma_2-\sigma_1)}{\sigma_2 } $$ 
and by the Young inequality
$$ \| u \|_{L^{p + \sigma_1}}^{p+\sigma_1} \leq  \frac{\sigma_2}{\sigma_1}
\| u \|_{L^{p+\sigma_2}}^{p + \sigma_2} + \frac {\sigma_2 - \sigma_1}{\sigma_2} \| u \|_{L^p}^p. $$
Hence, if 
$$ k\leq 0,\,\, \alpha \frac{p-2} {2} \leq a,\,\, b \frac{\sigma_1}{\sigma_2} \leq c \,\,\,{\rm and}\,\,\,
 b \frac {\sigma_2 - \sigma_1}{\sigma_2} \leq | k | $$
we derive that $ \dot W_p (u) \leq p k \| u \|_{L^p}^p$
and the conclusion follows from the Lemma 4.2. If $p = 2$ and $\Omega$ is bounded, by the Poincar\'e inequality,
we obtain the same conclusion under the conditions
$$ k > 0,\,\, b \frac{\sigma_1}{\sigma_2} \leq c \,\,\,{\rm and } \,\,\,
   b^+  \frac {\sigma_2 - \sigma_1}{\sigma_2} +   k < a  \left( \frac {1} {\omega_N} | \Omega | \right)^{2/N},$$
 with $b^+ = \max\{0, b\}$.

 \medskip
 \noindent 2. We now define the new functional:
 \begin{multline}\label{Vliap}
 V(u) := \frac {a} {2} \int_\Omega | \nabla u|^2 dx - \frac {b} {\sigma_1 +2} \int_\Omega | u |^{\sigma_1 + 2} dx\,  \\
   + \frac {c} {\sigma_2 +2} \int_\Omega | u |^{\sigma_2 + 2} dx        - \frac {k} {2} \int_\Omega | u |^2 dx. 
 \end{multline}
 
  \noindent It is clear that $V$ is a continuous real function on $H_0^1(\Omega)$. By interpolation
and the Young inequality, we have
\begin{equation}\label{Vliap8}
\| u \|_{L^{\sigma_1+2}}^{\sigma_1+2}  \leq
 \|u \|_{L^{\sigma_2 +2}}^\frac{\sigma_2+2}{\sigma_2} \| u \|_{L^2}^\frac{2 (\sigma_2 - \sigma_1)}{\sigma_2}  
  \leq \frac{\sigma_1}{\sigma_2} \| u \|_{L^{\sigma_2+2}}^{\sigma_2+2} +  \frac{\sigma_2 - \sigma_1}{\sigma_2}
  \| u \|_{L^2}^2
 \end{equation}
 Then we have $ V(u) \geq M \| u \|_{H^1},\, M > 0$, if 
 \begin{equation}\label{cond1}
  k < 0, \,\,\, \frac{b}{\sigma_1+2} \frac{\sigma_1}{\sigma_2} \leq \frac{c}{\sigma_2+2}\,\,\, {\rm  and}\,\,\,
  \frac{b}{\sigma_1+2} \frac{\sigma_2-\sigma_1}{\sigma_2} \leq \frac{|k|}{2}   
  \end{equation}
  or
   \begin{equation} \label{cond2}
   \begin{split}
  & k = 0, \,\,\, \frac{b}{\sigma_1+2} \frac{\sigma_1}{\sigma_2} \leq \frac{c}{\sigma_2+2}\,\,\, {\rm  and}\,\,\,
  \frac{b (\sigma_2-\sigma_1)}{(\sigma_1+2) \sigma_2} \leq \frac{a}{2}
 \left( \frac{1}{\omega_N} |\Omega| \right)^{-2/N}  \\
 & {\rm and } \,\,\Omega \,\, \hbox{\rm is a bounded domain}
 \end{split}
  \end{equation}
  
  \smallskip
   
\noindent  In addition, for any $u\in H_0^1(\Omega) \cap H^2 (\Omega)$ and $ h \in H_0^1(\Omega)$,
  we have $ V(u + h) = V(u) + L \cdot h + o (\| h \|_{H^1}) $, where
  $$ L \cdot h = - \Re \int_\Omega \left[ a \Delta \overline{u} - b | u |^\sigma \overline{u} 
  + k \overline{u} \right]  h \, dx. $$
  Therefore, for all $u = u(t) \in H_0^1(\Omega) \cap H^2 (\Omega)$,
  \begin{equation*}
  \begin{split}
  \dot V(u) =& - \int_\Omega  \left| a \Delta u + b | u |^{\sigma_1} u -  c | u |^{\sigma_2} u + k u \right|^2  dx \\
  & - \Re \int_\Omega \left( a \Delta u + b | u |^{\sigma_1} u - c | u |^{\sigma_2} u  \right)
   i \left( \alpha \Delta u + \beta | u |^{\sigma_1} u - \gamma | u |^{\sigma_2} u \right) dx \\
  & - \Re \int_\Omega i \,k \overline{u}  \left( \alpha \Delta u + \beta | u |^{\sigma_1} u
  - \gamma | u |^{\sigma_2} u \right) dx \\
  \end{split}
  \end{equation*}
  and, for $\frac {\alpha} {a} = \frac {\beta} {b} = \frac{\gamma}{c}$, we obtain
  \begin{equation}\label{Vliap2} 
  \dot V (u(t)) = - \int_\Omega \left| a \Delta u + b | u |^\sigma_1 u - c | u |^\sigma_2 u  + k u \right|^2  dx \leq 0,\,\, t > 0.
  \end{equation}
  Note that
  $$ \frac {1} {t} \,[V(S(t) u_0) - V(u_0)] = \dot V (S (t^*) u_0)$$
  for some $0 < t^* < t$ and so \eqref{Vliap2} is true for all $t \geq 0$. Hence, the functional $V$ is a
  Lyapunov function and, under the conditions \eqref{cond1}, \eqref{cond2},
   we have the stability in $H_0^1(\Omega)$ of the equilibrium solution $ u \equiv 0$.

  \noindent We prove now the asymptotic stability. We have
  \begin {equation}\label{Vliap3}
  \begin{split}
  & - \dot V (u) =  a^2 \int_\Omega | \Delta u|^2 dx + b^2 \int_\Omega | u |^{2 \sigma_1 + 2} dx   \\
    & + c^2 \int_\Omega | u |^{2 \sigma_2 + 2} dx + k^2 \int_\Omega | u |^2 dx 
   + 2 a b \,\Re \int_\Omega \Delta u | u |^{\sigma_1} \overline{u} \,dx  \\
   & -2 a c \,\Re \int_\Omega \Delta u | u |^{\sigma_2} \overline{u} \,dx
   + 2 a k \, \Re \int_\Omega \Delta u \overline{u} \,dx 
  + 2 b k\,  \int_\Omega | u |^{\sigma_1 + 2} \,dx  \\
   & - 2 c k\,  \int_\Omega | u |^{\sigma_2 + 2} dx 
   - 2 b c   \int_\Omega | u |^{\sigma_1 + \sigma_2 + 2}\, dx
  \end{split}
  \end{equation}
  
\noindent  Next one has the following estimates:
  
   \begin{equation*}
   \begin{split}
   \Re \int_\Omega \Delta u | u |^{\sigma_2} \overline{u} dx  & = -\int_\Omega | \nabla u |^2 | u |^{\sigma_2} dx \\
   & - \frac {\sigma_2} {2} \Re \int_\Omega | u |^{\sigma_2 -2} \nabla u \cdot ( \nabla u \overline{u}
   + u \nabla \overline{u} ) \overline{u} \, dx \\
   & = - \int_\Omega | \nabla u |^2 | u |^{\sigma_2} dx - \frac {\sigma_2} {2} \int_\Omega | \nabla u |^2 | u |^{\sigma_2} dx \\
   & \quad - \frac {\sigma_2} {2} \Re \int_\Omega | u |^{\sigma_2 -2} (\nabla u \cdot  \nabla u ) \overline{u}^2 dx  \\
   & \leq - \int_\Omega | \nabla u |^2 | u |^\sigma dx \\
   \end{split}
   \end{equation*}
   and so
   \begin{equation}\label{Vliap5}
   - 2 a c \,\Re \int_\Omega \Delta u | u |^{\sigma_2} \overline{u} \,dx
    \geq 2 a c \int_\Omega | \nabla u |^2 | u |^{\sigma_2} dx.
   \end{equation}
   Also
   \begin{equation}\label{Vliap6}
   2 a k\, \Re \int_\Omega \Delta u \overline{u} dx = - 2 a k \int_\Omega | \nabla u |^2 dx. 
   \end{equation}
   Since
   $$  \int_\Omega | u |^{\sigma_1 + \sigma_2 + 2}\, dx \leq \| u \|_{L^{2 \sigma_1+2}}^{\sigma_1+1}
   \| u \|_{L{2 \sigma_2+2}}^{\sigma_2+1}$$
   we obtain
   \begin{equation}\label{Vliap7}
   -2 b c  \int_\Omega | u |^{\sigma_1 + \sigma_2 + 2}\, dx \geq - b^2 \| u \|_{L^{2 \sigma_1+2}}^{2\sigma_1+2}
   - c^2  \| u \|_{L{2 \sigma_2+2}}^{2\sigma_2+2}
   \end{equation}
   and, if $ b \sigma_1/\sigma_2 < c$ and $ b (\sigma_2 - \sigma_1)/\sigma_2 < |k|/2$,
   it follows from \eqref{Vliap8} 
   \begin{equation}\label{Vliap9}
   2 | b k| \,\| u \|_{L^{\sigma_1+2}}^{\sigma_1+2} < 2 c | k |\, \| u \|_{L^{\sigma_2+2}}^{\sigma_2+2}
   + k^2 \| u \|_{L^2}^2
   \end{equation}
   Finally we remark that
   $$ \left|  \int_\Omega \Delta u | u |^{\sigma_1} \overline{u} dx \right| \leq
   (\sigma_1 +1) \int_\Omega | \nabla u |^2 | u |^{\sigma_1} dx  $$
   
  \noindent  and so, if we assume  $ | b | \, (\sigma_1 +1) < \min \{ c, |k| \} $, we get
   \begin{equation}\label{Vliap10}
   \begin{split}
  & \left|  2 a b \int_\Omega \Delta u | u |^{\sigma_1} \overline{u} dx \right| \\
  & \leq 2 a | b | (\sigma_1 +1) \int_\Omega | \nabla u |^2 dx + 2 a | b | (\sigma_1 +1) \int_\Omega | \nabla u |^2 
   | u |^{\sigma_2} dx  \\
   & < 2 a c \int_\Omega | \nabla u |^2 | u |^{\sigma_2} dx + 
   2 a | k | \int_\Omega | \nabla u |^2 dx
   \end{split}
   \end{equation}
   If $k < 0$, is now clear that the asymptotic stability of $u\equiv 0$ follows from
   \eqref{Vliap3} and \eqref{Vliap5}, \eqref{Vliap6}, \eqref{Vliap7}, \eqref{Vliap9}, \eqref{Vliap10}.
   
   \noindent With $k=0$ and $\Omega $ a bounded domain, we estimate
   $$ \int_\Omega | \nabla u|^2 dx \leq \left| \int_\Omega \Delta u \overline{u} \,dx \right| \leq
  \left( \int_\Omega | \Delta u |^2 dx \right)^{1/2} \left( \int_\Omega | u |^2 dx \right)^{1/2}
  $$
  and by the Poincar\'e inequality,
  $$ \left( \int_\Omega | \Delta u |^2 dx \right)^{1/2} \geq \left(\frac {1} {\omega_N} | \Omega | \right)^{- 1/N}
  \left( \int_\Omega | \nabla u |^2 dx \right)^{1/2}. $$
  Hence
  \begin{equation}\label{Vliap4}
  a^2 \int_\Omega | \Delta u |^2 dx \geq a^2 \left(\frac {1} {\omega_N} | \Omega | \right)^{- 2/N}
   \int_\Omega | \nabla u |^2 dx.
   \end{equation}
   Since $k=0$, it is sufficient to estimate the fifth term in the r.h.s of \eqref{Vliap3} with
   $b > 0$. From the estimations \eqref{Vliap5}, \eqref{Vliap10} and \eqref{Vliap4}  we must require
   $$ b (\sigma_1 +1 ) \leq c,\quad {\rm and}\quad  b (\sigma_1 +1) 
   < \frac{a}{2} \left(\frac {1} {\omega_N} | \Omega | \right)^{- 2/N}$$
   and we note that this second condition imply the last stability condition in \eqref{cond2}. The proof
   is now complete.
    \end{proof}
    \vskip 10pt
   
   \section{ Stability of some time periodic solutions of \eqref{gCGL}.}
   
   \vskip 10pt
   
   Consider the \eqref{gCGL} equation on a bounded domain $\Omega$ with the Neumann condition
   on the boundary. We study now the stability of some particular time periodic solutions. Le be
   $\vartheta (t) $ a $T$-periodic solution of the ordinary differential equation \eqref{ord-GL},
   \begin{equation*}
   \dot{u} = (b + i \beta) | u |^{\sigma_1} u - (c + i \gamma) | u |^{\sigma_2} u + ku 
   \end{equation*}
   associated to the \eqref{gCGL} equation. \smallskip
   
   \begin{proof}[Proof of Theorem 1.10] First we linearise the \eqref{gCGL} equation around the $T$-periodic solution
   	$\Theta (x, t) \equiv \vartheta (t)$. We obtain the linear variational equation
   	\begin{equation}\label{eq-variational}
   	\partial_t v = A_N v + B(t) v
   	\end{equation}
   	where $A_N = (a + i \alpha) \Delta$ denote the Neumann operator. If we set $ v = v_1 + i v_2$
   	and $\vartheta = \vartheta_1 + i \vartheta_2$ we have
   	
   	\begin{equation}\label{eq-B1}
   	\begin{split}
   	\Re (B(t) v )   = \,\,&  b | \vartheta |^{\sigma_1} v_1  - \beta  | \vartheta |^{\sigma_1} v_2
   	+ b\, \sigma_1  | \vartheta |^{\sigma_1-2} \vartheta_1 \Re (\overline{\vartheta} \,v ) \\
   	& - \beta \sigma_1  | \vartheta |^{\sigma_1-2} \vartheta_2 \Re (\overline{\vartheta} \,v ) -
   	c\, | \vartheta |^{\sigma_2} v_1 + \gamma | \vartheta |^{\sigma_2} v_2    \\
   	& - c \,\sigma_2  | \vartheta |^{\sigma_2-2} \vartheta_1 \Re (\overline{\vartheta} \,v )
   	+ \gamma \sigma_2  | \vartheta |^{\sigma_2-2} \vartheta_2 \Re (\overline{\vartheta} \,v ) 
   	\end{split}
   	\end{equation}
   	
   	\begin{equation}\label{eq-B2}
   	\begin{split}
   	\Im (B(t) v )   = \,\,&  \beta | \vartheta |^{\sigma_1} v_1  +b  | \vartheta |^{\sigma_1} v_2
   	+ b\, \sigma_1  | \vartheta |^{\sigma_1-2} \vartheta_2 \Re (\overline{\vartheta} \,v ) \\
   	& + \beta \sigma_1  | \vartheta |^{\sigma_1-2} \vartheta_1 \Re (\overline{\vartheta} \,v ) -
   	c\, | \vartheta |^{\sigma_2} v_2 - \gamma | \vartheta |^{\sigma_2} v_1    \\
   	& - c \,\sigma_2  | \vartheta |^{\sigma_2-2} \vartheta_2 \Re (\overline{\vartheta} \,v )
   	- \gamma \sigma_2  | \vartheta |^{\sigma_2-2} \vartheta_1 \Re (\overline{\vartheta} \,v ) 
   	\end{split}
   	\end{equation}
   	Notice that $B(t)$ is $T$-periodic. 
   	
   	\noindent Now, let $R(t, s)$ the evolution operator for \eqref{eq-variational}, i.e.
   	$$ R(t, s) v_0 = v(t; s, v_0) $$
   	is the solution of \eqref{eq-variational} with initial data, $v(s) = v_0$, and recall that the eigenvalues 
   	of the {\it period map}, $U_0 = R(T, 0)$, are  the characteristic multipliers. Since $A_N$ has compact
   	resolvent, $U_0$ is compact and so, the spectrum $\sigma (U_0)\backslash \{0\}$ is entirely
   	composed by characteristic multipliers (see \cite[pg. 197]{Henry}). Next, we prove the following result:
   	
   	\noindent The characteristic multipliers of \eqref{eq-variational} are the multipliers of the planar
   	system
   	\begin{equation}\label{eq-variational2}
   	\dot v = - \lambda v + B(t) v
   	\end{equation}
   	for any $\lambda$, eigenvalue of the Neumann operator $-A_N = -(a + i \alpha ) \Delta$.
   	
   	\noindent In fact, let $ \tilde R (t, s)$ be the evolution operator for the planar system $\dot v = B(t) v$. By the
   	Floquet representation we have
   	$$ \tilde U_0 := \tilde R (T, 0) = P(T) e^{ C T} P(T)^{-1} $$
   	where $C$ is a constant matrix and $P(T)$ is an invertible matrix. Then we obtain
   	\begin{equation*}
   	\begin{split}
   	U_0 = R (T, 0) = & \,\, e^{A_N T} \tilde R (T, 0) \\
   	= & \,\, e^{A_N T} P(T) e^{ C T} P(T)^{-1} = P(T) e^{(A_N + C)T} P(T)^{-1} 
   	\end{split}
   	\end{equation*}
   	and so the eigenvalues of $U_0$ are the eigenvalues of $ e^{(A_N + C)T} $, i.e. the characteristic
   	multipliers of \eqref{eq-variational} are those of \eqref{eq-variational2}. 
   	
   	\noindent Denote this multipliers by $\mu_j, (j=1, 2)$. It is well known that $\mu_j$ must meet
   	the condition (see, e.g. \cite{Coddington})
   	\begin{equation}\label{prod-caract}
   	\mu_1 \mu_2 = \exp \left( \int_0^T {\rm Tr} (-\lambda I + B(t) )\right) 
   	\end{equation}
   	We consider now the two cases stated in the theorem:
   	
   	\smallskip
   	
   	\noindent 1. In \eqref{gCGL} equation let $c= 0$
   	and assume $ b k < 0$. Take the $T_1$-periodic solution $P(x, t) \equiv p(t)$ for all $x\in \Omega$. We obtain, 
   	for each $\lambda$ eigenvalue of $- \Delta$ with the Neumann condition,
   	\begin{equation}
   	\begin{split}
   	\mu_1 \mu_2 = &\,\, \exp \left( \int_0^{T_1} b | p(t) |^{\sigma_1} (2 + \sigma_1) + 2 k - 2 \lambda \,\, dt \right) \\
   	= & \exp \left( \int_0^{T_1} \sigma_1 b | p(t)|^{\sigma_1}  - 2 \lambda \,\, dt \right)
   	\end{split}
   	\end{equation}
   	since $ b | p(t) |^{\sigma_1} + k = 0 $ for all $ t\in [0, T_1]$ ( recall that the $T_1$-periodic solution $p(t)$
   	has his orbit in the circle $ | z | = r_1$, with $b r_1^{\sigma_1} + k = 0$ ). 
   	If $b < 0$ it is clear that
   	$$ \mu_1 \mu_2 = \exp (-k \sigma_1 T_1 - 2 \lambda T_1) < \exp (- k \sigma_1 T_1)  < 1$$
   	for all $\lambda \in \sigma(- A_N) $, which implies the asymptotic stability of $P(x, t)$ (see \cite{Henry}, Th.8.2.3).
   	
   	\noindent If $b > 0$ (and $k < 0$), easily we find the instability of the solution $p(t)$ of \label{ord2-GL}
   	(and so the instability of $P(x, t)$). In fact, multiply \eqref{ord-GL} by $\bar{u}$ and take the real part. We obtain
   	$$ \frac{d}{dt} | p |^2 = 2 b | p |^{\sigma_1+2} + 2 k | p |^2 $$
   	and the solution $| p(t) |^2 $ with initial data $| p(0) |^2 = r_1^2+\varepsilon, $ ($\varepsilon > 0$),
   	blow up in a finite time, since $ 2 b r_1^{\sigma_1+2} + 2 k r_1^2 = 0$.

   	\medskip
   	
   	\noindent 2. Assume now $k=0$ and $bc > 0$. Take the $T_2$-periodic solution $Q(x, t) \equiv q(t)$ and 
   	recall that $q(t)$ has his orbit in the circle $ | z | = r_2$, with $b - c r_2^{\sigma_2-\sigma_1} = 0$. We have
   	\begin{equation*}
   	\begin{split}
   	\mu_1 \mu_2 = &  \exp \left( \int_0^{T_2} 2 b | q(t) |^{\sigma_1} + b \sigma_1   | q(t) |^{\sigma_1}
   	- 2 c | q(t) |^{\sigma_2} - c \sigma_2 | q(t) |^{\sigma_2}  - 2 \lambda \,\, dt \right) \\
   	= & \exp \left(\int_0^{T_2} b \sigma_1   | q(t) |^{\sigma_1} - c \sigma_2 | q(t) |^{\sigma_2}  - 2 \lambda \,\, dt \right) \\
   	= & \exp\, (-c (\sigma_2 - \sigma_1) r_2^{\sigma_2} T_2 - 2 \lambda T_2)
   	\end{split}
   	\end{equation*}
   	We pursued just like before: if $c > 0$ (and so $ b > 0$) we have
   	$  \mu_1 \mu_2 = \exp ( -c (\sigma_2 - \sigma_1) r_1^{\sigma_2} T_2 - 2 \lambda T_2) 
   	\leq  \exp ( -c (\sigma_2 - \sigma_1) r_1^{\sigma_2} T_2 ) < 1$ for all $\lambda \in \sigma(- A_N) $
   	which proves the the asymptotic stability of $Q(x, t)$.
   	
   	\noindent If $c < 0$ and $b < 0$, from \eqref{ord-GL} we derive
   	$$ \frac{d}{dt} | q |^2 = 2 b | q |^{\sigma_1+2} - 2 c | q |^{\sigma_2+2} $$
   	which implies the blow-up in a finite time since 
   	$2 b r_2^{\sigma_1+2} - 2 c  r_2^{\sigma_2+2} = 0$. In particular we have prove, in this case,
   	the instability of $Q(x, t)$.

   \end{proof}

   \section{Acknowledgements}
  S. Correia was partially supported by Funda\c c\~ao para a Ci\^encia e Tecnologia, through the grant UIDB/MAT/04459/2020.
  M. Figueira was partially supported by Funda\c c\~ao para a Ci\^encia e Tecnologia, through the grant UIDB/04561/2020.

 \bigskip
 \bigskip
 
 \normalsize
 
 \begin{center}
	{\scshape Sim\~ao Correia}\\
{\footnotesize
	Center for Mathematical Analysis, Geometry and Dynamical Systems,\\
	Department of Mathematics,\\
	Instituto Superior T\'ecnico, Universidade de Lisboa\\
	Av. Rovisco Pais, 1049-001 Lisboa, Portugal\\
	simao.f.correia@tecnico.ulisboa.pt
}
  \bigskip
 
 	 	{\scshape M\'ario Figueira}\\
 	{\footnotesize
 		CMAF-CIO, Universidade de Lisboa\\
Edif\'{\i}cio C6, Campo Grande\\
1749-016 Lisboa, Portugal\\
\email{msfigueira@fc.ul.pt}
}

 \end{center} 
 
 \end{document}